\documentclass[11pt]{amsart}
\usepackage{a4wide}
\usepackage[all]{xy}

\newcommand{\e}{\varepsilon}
\newcommand{\w}{\omega}
\newcommand{\defeq}{\overset{\mbox{\tiny\sf def}}=}
\newcommand{\F}{\mathcal F}
\newcommand{\IN}{\mathbb N}

\title{On symmetrizability and perfectness of second-countable spaces}
\author{Iryna Banakh, Taras Banakh and Lidiya Bazylevych}
\address{Ya. Pidstryhach Institute for Applied Problems of Mechanics and Mathematics, Naukova 3b, Lviv, Ukraine}
\email{ibanakh@yahoo.com}
\address{Faculty of Mechanics and Mathematics, Ivan Franko National University of Lviv, Ukraine}
\email{t.o.banakh@gmail.com}
\email{izar@litech.lviv.ua}

\newtheorem{theorem}{Theorem}
\newtheorem{corollary}[theorem]{Corollary}
\newtheorem{proposition}[theorem]{Proposition}
\newtheorem{problem}[theorem]{Problem}
\newtheorem{question}[theorem]{Question}
\theoremstyle{definition}
\newtheorem{example}[theorem]{Example}

\begin{document}
\begin{abstract}
A symmetrizability criterion of Arhangelskii implies that a second-countable Hausdorff space is symmetrizable if and only if it is perfect. We present an example of a non-symmetrizable second-countable submetrizable space of cardinality $\mathfrak q_0$ and study the smallest possible cardinality $\mathfrak q_i$ of a non-symmetrizable second-countable $T_i$-space for $i\in\{1,2\}$. 
\end{abstract}
\maketitle

Let us recall that a function $d:X\times X\to[0,\infty)$ on a set $X$ is a {\em metric} if for every points $x,y,z\in X$ the following conditions are satisfied:
\begin{enumerate}
\item $d(x,y)=0$ if and only if $x=y$;
\item $d(x,y)=d(y,x)$;
\item $d(x,z)\le d(x,y)+d(y,z)$.
\end{enumerate}
A function $d:X\times X\to[0,\infty)$ is called a {\em premetric} (resp. a {\em symmetric}) on $X$ if it satisfies  the condition (1) (resp. the conditions (1) and (2)). 

We say that the topology of a topological space $X$ is {\em generated by a premetric} $d$ if a subset $U\subseteq X$ is open if and only if for every $x\in U$ there exists $\e>0$ such that $B_d(x;\e)\subseteq U$ where $B_d(x;\e)\defeq\{y\in X:d(x,y)<\e\}$ is the $\e$-ball centered at $x$. 

A topological space is {\em symetrizable} (resp. {\em metrizable}) if its topology is generated by some symmetric (resp. metric). Symmetrizable spaces satisfy the separation axiom $T_1$.  

By the classical Urysohn Metrization Theorem \cite[4.2.9]{Eng}, each second-countable regular space is metrizable.

In this paper we address the following question (asked by the first author at {\tt Mathoverflow} \cite{MO}).

\begin{problem}\label{prob1} Is every second-countable Hausdorff space symmetrizable?
\end{problem}

To our surprise we have discovered that Problem~\ref{prob1} has negative answer, contrary to Theorem 2.9 \cite{Ar} of Arhangelski claiming that a first-countable $T_1$ space with a $\sigma$-discrete network is symmetrizable. In fact, the proof of this theorem works only under an additional restriction that the $\sigma$-discrete network consists of closed sets. Let us recall that a family $\mathcal F$ of subsets of a topological space $X$ is called a {\em network} for $X$ if for every open set $U\subseteq X$ and point $x\in U$ there exists a set $F\in\mathcal F$ such that $x\in F\subseteq X$. A network $\F$ is {\em closed} if every set $F\in\mathcal F$ is closed in $X$. So, the correct version of Arhangelskii's Theorem 2.9 in \cite{Ar} reads as follows.

\begin{theorem}[Arhangelskii] Every first-countable $T_1$  space with a $\sigma$-discrete closed network is symmetrizable.
\end{theorem}

This theorem implies

\begin{corollary}\label{c:Ar} Every first-countable $T_1$-space with a countable closed network is symmetrizable.
\end{corollary}

We shall apply this corollary to prove that for second-countable Hausdorff spaces the symmetrizability is equivalent to the perfectness. 

A topological space $X$ is called {\em perfect} if every closed subset $F$ of $X$ is of type $G_\delta$, i.e., $F$ is the intersection of countably many open sets. The following proposition is known, see the discussion before Theorem~9.8 in \cite{Grue}.

\begin{proposition}\label{p:s=>p} Every first-countable symmetrizable Hausdorff space $X$  is perfect.
\end{proposition}

\begin{proof} We present a short proof for the convenience of the reader. Let $d$ be a symmetric generating the topology of $X$. First we show that for every $x\in X$ and $\e>0$ the point $x$ is contained in the interior $B_d(x;\e)^\circ$ of the ball $B_d(x;\e)$. Indeed, in the opposite case we can use the first-countability of $X$ and find a sequence $S=\{x_n\}_{n\in\w}\subseteq X\setminus B_d(x;\e)$ that converges to $x$. The Hausdorff property of $X$ implies that the compact subset $K=\{x\}\cup S$ is closed in $X$ and hence for every $y\in X\setminus K$ there exists $\e_y$ such that $B_d(y;\e_y)\cap S\subseteq B_d(y;\e_y)\cap K=\emptyset$. Since also $B_d(x;\e)\cap S=\emptyset$, the set $S$ is closed in $X$, which is not possible as the sequence $(x_n)_{n\in\w}$ converges to $x\notin S$. This contradiction shows that $x$ is an interior point of the ball $B_d(x;\e)$.

Now we are ready to prove that $X$ is perfect. Given any closed set $F\subseteq X$, for every $n\in\IN$ consider the open neighborhood $U_n\defeq\bigcup_{x\in F}B_d(x;\frac1n)^\circ$ of $F$ in $X$ and observe that $F=\bigcap_{n\in\w}U_n$.
\end{proof}

\begin{proposition}\label{p:p=>s} Each perfect second-countable $T_1$-space $X$ is symmetrizable.
\end{proposition}

\begin{proof} Let $\mathcal B$ be a countable base of the topology of $X$. By the perfectness of $X$, every open set $B\in\mathcal B$ is equal to the union $\bigcup\mathcal F_B$ of a countable family $\F_B$ of closed sets in $X$. Then $\F\defeq\bigcup_{B\in\mathcal B}\F_B$ is a countable closed network for $X$. By Corollary~\ref{c:Ar}, the space $X$ is symmetrizable.
\end{proof}

Propositions~\ref{p:s=>p} and \ref{p:p=>s} imply the following criterion.

\begin{theorem} A second-countable Hausdorff space is symmetrizable if and only if it is perfect.
\end{theorem}

Next, we prove that second-countable (submetrizable) $T_1$-spaces of sufficiently small cardinality are symmetrizable.

A topological space is {\em submetrizable} if it  admits a continuous metric.
Each submetrizable space is {\em functionally Hausdorff} in the sense that for any distinct elements $x,y\in X$ there exists a continuous function $f:X\to\mathbb R$ such that $f(x)\ne f(y)$. By \cite{BB}, a second-countable space is submetrizable if and only if it is functionally Hausdorff.

A topological space $X$ is called a {\em $Q$-space} if every subset of $X$ is of type $G_\delta$ in $X$. Every $Q$-space is perfect.

Let $\mathfrak q_0$ be the smallest cardinality of a second-countable metrizable space which is not a $Q$-space.  For properties of the cardinal $\mathfrak q_0$, see \cite[\S4]{Miller}, \cite{Brendle}, \cite{BMZ}. 

For $i\in\{1,2,3\}$, let $\mathfrak q_i$ be the smallest cardinality of a second-countable $T_i$-space which is not a $Q$-space. It is clear that $\w_1\le \mathfrak q_1\le \mathfrak q_2\le\mathfrak q_3=\mathfrak q_0\le\mathfrak c$, where $\mathfrak c$ stands for the cardinality of continuum. By \cite{BB}, $\mathfrak p\le\mathfrak q_1$, where $\mathfrak p$ is the smallest cardinality of a subfamily $\mathcal B\subseteq[\w]^\w$ such that for every finite subfamily $\F\subseteq\mathcal B$ the intersection $\bigcap\F$ is infinite but for every infinite set $I\subseteq\w$ there exists a set $F\in\F$ such that $I\cap F$ is finite. It is well-known (see \cite{Blass} or \cite{Vaughan})  that $\mathfrak p=\mathfrak c$ under Martin's Axiom. By \cite{BB}, every submetrizable space of cardinality $<\mathfrak q_0$ is a $Q$-space. This fact combined with Proposition~\ref{p:p=>s} implies

\begin{proposition}\label{p:q0} Every second-countable submetrizable space of cardinality $<\mathfrak q_0$ is symmetrizable.
\end{proposition}

Since functionally Hausdorff second-countable spaces are submetrizable, Proposition~\ref{p:q0} implies another criterion of symmetrizability.

\begin{proposition}\label{p:q0'} Every second-countable functionally Hausdorff space of cardinality $<\mathfrak q_0$ is symmetrizable.
\end{proposition}

Proposition~\ref{p:p=>s} combined with the definition of the cardinals $\mathfrak q_i$ implies the following semimetrizability criterion.

\begin{proposition}\label{p:qi} Let $i\in\{1,2\}$. Every second-countable $T_i$-space $X$ of cardinality $|X|<\mathfrak q_i$ is symmetrizable.
\end{proposition}

Since $\mathfrak p=\mathfrak q_1=\mathfrak c$, we obtain the following corollary. 

\begin{corollary} Under Martin's Axiom, every second-countable $T_1$-space of cardinality $<\mathfrak c$ is symmetrizable.
\end{corollary}

Finally, we show that the cardinals $\mathfrak q_0$ and $\mathfrak q_2$ in Propositions~\ref{p:q0}, \ref{p:q0'}, \ref{p:qi} are the best possible.

\begin{example} There exists a second-countable submetrizable space of cardinality $\mathfrak q_0$ which is not symmetrizable.
\end{example}

\begin{proof} By the definition of the cardinal $\mathfrak q_0$, there exists a second-countable metrizable space $X$, which is not a $Q$-space. Then $X$ contains a subset $A$ which is not of type $G_\delta$ in $X$. Let $\tau'$ be the topology on $X$ generated by the subbase $\tau\cup\{A\}$ where $\tau$ is the topology of the space $X$.  Since $\tau\subseteq\tau'$, the space $X'$ is submetrizable. Assuming that $X'$ is perfect, we conclude that the closed set $A$ is equal to the intersection $\bigcap_{n\in\w}W_n$ of some open sets $W_n\in\tau'$. By the choice of the topology $\tau'$, for every $n\in\w$ there exists open sets $U_n,V_n\in \tau$ such that $W_n=U_n\cup(V_n\setminus A)$. It follows from $A\subseteq W_n=U_n\cup(V_n\setminus A)$ that $A=A\cap W_n=A\cap U_n\subseteq U_n$.
$$A=\bigcap_{n\in\w}W_n=A\cap\bigcap_{n\in\w}W_n=\bigcap_{n\in\w}(A\cap W_n)=\bigcap_{n\in\w}(A\cap U_n)\subseteq \bigcap_{n\in\w}U_n\subseteq \bigcap_{n\in\w}W_n=A$$
and hence $A=\bigcap_{n\in\w}U_n$ is a $G_\delta$-set in $X$, which contradicts the choice of $A$. This contradiction shows that the submetrizable second-countable space $X'$ is not perfect. By Proposition~\ref{p:s=>p}, $X'$ is not symmetrizable.
\end{proof}

By analogy, we can prove that the cardinal $\mathfrak q_2$ in Proposition~\ref{p:qi} is the best possible.

\begin{example} There exists a second-countable Hausdorff space of cardinality $\mathfrak q_2$ which is not symmetrizable.
\end{example}

However this argument does not work for the cardinal $\mathfrak q_1$ (because Proposition~\ref{p:p=>s} is applicable only for Hausdorff spaces).

\begin{problem} Is $\mathfrak q_1$ equal to the smallest cardinality of a second-countable $T_1$-space which is not symmetrizable?
\end{problem}

\begin{problem} Is $\mathfrak q_1=\mathfrak q_2=\mathfrak q_0$?
\end{problem}

By Proposition~\ref{p:s=>p}, every symmetrizable first-countable Hausdorff space is perfect. On the other hand, by \cite{Bonnett}, \cite{DGN}, \cite{Step}, there exists a non-perfect symmetrizable Hausdorff (even regular) spaces. However those examples are not first-countable.

\begin{question}\label{q:p} Is every second-countable symmetrizable space perfect?
\end{question}

\begin{example} Consider the set $X=\{-\frac1n:n\in\mathbb N\}\cup\{0\}\cup\{\tfrac1n:n\in\mathbb N\}$ endowed with the symmetric
$$d(x,y)=\begin{cases}\max\{x,y\}&\text{if $\max\{x,y\}>0$ and $\min\{x,y\}<0$};\\
\lvert x-y\rvert&\text{if $\max\{x,y\}\le 0$};\\
0&\text{if $x=y$};\\
1&\text{otherwise}.
\end{cases}
$$
This symmetric generates a first-countable non-Hausdorff topology in which the unit ball $B_d(0,1)=\{-\frac1n:n\ge 2\}\cup\{0\}$ is nowhere dense. So, the proof of the perfectness of symmetrizable first-countable Hausdorff spaces from Proposition~\ref{p:s=>p} does not work in this case. Nonetheless, the symmetrizable space $X$ is countable and hence perfect, so this example does not answer Question~\ref{q:p}.
\end{example}

\end{document}